\documentclass{amsart}

\newtheorem{theorem}{Theorem}[section]
\newtheorem{lemma}[theorem]{Lemma}
\newtheorem{proposition}[theorem]{Proposition}
\theoremstyle{definition}
\newtheorem{definition}[theorem]{Definition}

\newtheorem{remark}[theorem]{Remark}

\usepackage{appendix,tikz,dsfont,fixmath}
\usepackage{nameref}
\usepackage{enumitem}
\usepackage{color}
\usepackage{hyperref}

\theoremstyle{rem}

\numberwithin{equation}{section}

\usepackage{tikz}
\usetikzlibrary{positioning, shapes.geometric}
\usetikzlibrary{shapes,arrows}
\usepackage{amsmath,amssymb}
\usepackage{mathtools}

\newcommand{\R}{\mathbb{R}}

\newcommand{\LapN}{\Delta_N}
\newcommand{\pa}{\partial}

\newcommand{\Ai}{\mathrm{Ai}}
\newcommand{\Bi}{\mathrm{Bi}}

\begin{document}
\title[Neumann Schr\"odinger Flows in Cylindrical Domains]{Neumann Schr\"odinger Estimates in Cylindrical Domains}

\author{MEAS Len}
\address{Department of Mathematics, Royal University of Phnom Penh, Phnom Penh, Cambodia}
\curraddr{Department of Mathematics, Royal University of Phnom Penh, Phnom Penh, Cambodia}
\email{meas.len@rupp.edu.kh}






\begin{abstract}

We establish sharp local-in-time dispersive estimates and global Strichartz estimates for the semiclassical Schr\"odinger equation in a three-dimensional cylindrical domain
\(\Omega\subset\mathbb{R}^{3}\) with homogeneous Neumann boundary conditions. The partially flat boundary leads to a degeneracy of curvature in the longitudinal direction and produces persistent glancing interactions that distinguish the cylindrical setting from strictly convex domains. At the spectral level, the Neumann condition is reflected in the appearance of the zeros of the Airy derivative in place of the Airy zeros associated with Dirichlet boundary conditions.

Using semiclassical microlocal analysis, oscillatory integral estimates, and an Airy--Poisson summation formula adapted to Neumann reflection, we obtain uniform dispersive bounds for the localized Schr\"odinger propagator. The quadratic structure of the Schr\"odinger phase yields a uniform treatment of the low-frequency regime and, combined with the boundary-layer analysis, leads to global Strichartz estimates with derivative loss
\[
\rho(q)=\frac{3}{2}\left(\frac12-\frac1q\right).
\]
As an application, we prove local well-posedness in \(H^{s}(\Omega)\), \(s>1\), for the cubic nonlinear Schr\"odinger equation with either focusing or defocusing sign and homogeneous Neumann boundary conditions. These results quantify the effect of partially flat boundary geometry on Schr\"odinger dispersion and provide a sharp linear framework for the corresponding nonlinear dynamics.

\end{abstract}

\maketitle

\section{Introduction and Formulation of the Cylindrical Model}
\label{sec:intro}

The analysis of dispersive partial differential equations on domains with
boundary is a fundamental problem in microlocal analysis and partial
differential equations. In the Euclidean setting $\mathbb{R}^{d}$, the
linear Schr\"odinger flow enjoys the classical dispersive decay associated
with its nondegenerate quadratic phase, which yields the scale-invariant
Strichartz estimates without loss of derivatives
\cite{KeelTao1998,Tao2006}. In the presence of a boundary, however,
reflections and glancing interactions may substantially alter this
mechanism. Concentration near reflected rays and the formation of caustics
can lead to losses in dispersive decay and, consequently, in the associated
Strichartz estimates
\cite{BGT2004,BSSSch,MS1978,MS1982}.

The influence of boundary geometry is particularly pronounced near
glancing trajectories. Microlocal analysis of boundary value problems
shows that the interaction of singularities with the boundary is governed
by a characteristic reflection and glancing structure
\cite{Melrose1975,MS1978,Taylor1976}. In strictly convex geometries,
Airy-type asymptotics describe the transition between reflected and
glancing regimes and play a central role in the analysis of dispersive
and Strichartz estimates
\cite{ILP2014,Ivanovici2010,Ivanovici2010Sc,Ivanovici2023}.
These phenomena have important consequences for the dispersive behavior
of the Schr\"odinger propagator and for the corresponding Strichartz
theory. In particular, recent semiclassical analysis of strictly convex
models exhibits a loss in the fixed-time dispersive estimate associated
with swallowtail-type caustics
\cite{ILLP2023, Ivanovici2023}.

In contrast, flat boundaries admit substantially simpler reflection
mechanisms, often reducible locally to reflection or even extension
arguments. The intermediate situation of a boundary with mixed curvature,
containing both curved and flat principal directions, requires a different
analysis. The corresponding geometry is not covered directly by either
the uniformly strictly convex theory or the completely flat case.

In this paper we study a three-dimensional cylindrical model in which one
principal curvature vanishes identically along the longitudinal direction,
while the transverse principal curvature remains strictly positive. This
geometry is therefore neither strictly convex nor completely flat. The
resulting degeneracy produces persistent glancing interactions in the
longitudinal direction and leads to an anisotropic dispersive mechanism
that is not captured by the standard strictly convex theory.

To formalize this configuration near a characteristic glancing region, we
use the half-space model
\begin{equation}
    \Omega
    =
    \left\{
    (x,y,z)\in\mathbb{R}^{3}:x\geq 0
    \right\},
\end{equation}
with boundary
\[
    \partial\Omega=\{x=0\}.
\]
Here $x$ is the normal variable, $y$ is the tangential direction in which
the boundary curvature is nonzero, and $z$ denotes the flat longitudinal
direction. We equip $\Omega$ with the Friedlander-type metric
\begin{equation}
    g=dx^{2}+\frac{1}{1+x}\,dy^{2}+dz^{2}.
\end{equation}
This model retains the leading-order coupling between the normal distance
$x$ and the curved tangential direction $y$, while the $z$-direction
remains translation invariant. The associated Laplace operator
is
\begin{equation}
    \LapN
    =
    \partial_x^{2}
    +
    (1+x)\partial_y^{2}
    +
    \partial_z^{2}.
\end{equation}
Thus $\LapN$ provides a localized normal form for the boundary geometry
relevant to the glancing analysis; its two-dimensional $(x,y)$ component
is the standard Friedlander-type model used in the study of boundary
dispersion
\cite{ILP2014,Ivanovici2010}.

We consider the initial-boundary value problem
\begin{equation}
\label{eq:SchrModel}
\begin{cases}
\left(i\partial_t+\LapN\right)u(t,x,y,z)=0,
    &(t,x,y,z)\in\mathbb{R}\times\Omega,\\[1mm]
u(0,x,y,z)=\delta_a(x,y,z),
    &a>0,\\[1mm]
\left.\partial_xu\right|_{x=0}=0,
    &(t,y,z)\in\mathbb{R}\times\partial\Omega,
\end{cases}
\end{equation}
where
\[
    \delta_a(x,y,z)
    =
    \delta(x-a)\delta(y)\delta(z)
\]
is a point source located at distance $a$ from the boundary. The parameter
$a$ allows us to probe the boundary layer in which high-frequency
reflection and glancing effects become dominant.

The Neumann condition introduces an additional spectral feature that is
essential to the analysis. After semiclassical Fourier reduction in the
tangential variables, the normal spectral problem is governed by an Airy
equation. For Dirichlet boundary conditions, the spectral parameters are
determined by the zeros of the Airy function, whereas the Neumann condition
selects the zeros of its derivative. These standard Airy asymptotics and
zero distributions underlie the resulting boundary spectral expansions
\cite{Olver1997}.

More precisely, after semiclassical rescaling and Fourier transformation
in $(y,z)$, the normal operator takes the form
\begin{equation}
    P_h(\eta,\zeta)
    =
    -h^{2}\partial_x^{2}
    +
    (1+x)\eta^{2}
    +
    \zeta^{2}.
\end{equation}
The corresponding eigenvalue problem reduces, after an Airy scaling, to
the quantization condition
\[
    \operatorname{Ai}'(-\omega_k)=0
\]
in the Neumann case. The associated Airy eigenfunctions and their
Poisson-summed spectral representation provide the basic building blocks
for the semiclassical propagator.

The principal difficulty is the simultaneous presence of a curved
direction and a flat direction. The longitudinal degeneracy prevents a
direct application of estimates designed for uniformly strictly convex
boundaries and produces a persistent concentration mechanism. Our analysis
combines semiclassical microlocalization, stationary-phase estimates, and
an Airy--Poisson summation formula adapted to Neumann reflection. This
allows us to separate the ordinary dispersive contribution from the
boundary-induced glancing contribution and to obtain uniform estimates
through the relevant boundary layers.

The main objectives are threefold:
\begin{enumerate}[label=(\alph*)]
    \item to construct a semiclassical representation of the Neumann
    Schr\"odinger propagator that remains uniform across the transition
    between the curved glancing sector and the degenerate longitudinal
    direction;

    \item to establish sharp local-in-time dispersive estimates and to
    identify the deterioration of the decay caused by the flat direction,
    namely the transition from the strictly convex behavior
    $O(|t|^{-5/4})$ to the cylindrical rate
    $O(|t|^{-3/4})$;

    \item to derive global Strichartz estimates with derivative loss
    \[
        \rho(q)
        =
        \frac{3}{2}
        \left(
        \frac12-\frac1q
        \right),
    \]
    and to apply these estimates to the cubic nonlinear Schr\"odinger
    equation with homogeneous Neumann boundary conditions, obtaining local
    well-posedness in $H^{s}(\Omega)$ for $s>1$.
\end{enumerate}

The paper is organized as follows. Section~\ref{sec:spectral} develops the
semiclassical reduction and Neumann Airy spectral problem;
Section~\ref{sec:poisson} proves the Airy--Poisson summation formula;
Section~\ref{sec:microlocal} introduces the microlocal frequency
localization and establishes the block-dispersive bounds;
Section~\ref{sec:dispersive} derives the global dispersive estimates and
boundary-layer bounds; Section~\ref{sec:strichartz} establishes the global
Strichartz estimates; and Section~\ref{sec:nls} treats the cubic
Schr\"odinger equation.

\section{Semiclassical Parametrix and Neumann Spectral Analysis}
\label{sec:spectral}

To construct a global representation of the localized Green function associated with problem \eqref{eq:SchrModel}, we exploit the explicit translation and metric invariance of the Friedlander model operator $\pa_x^2 + (1+x)\pa_y^2 + \pa_z^2$ with respect to the tangential variables $(y, z)$. Applying the partial Fourier transform in these boundary fields maps the multi-dimensional partial differential equation into a direct integral of decoupled one-dimensional fiber operators parameterized by the dual frequency momentum indices $(\eta, \zeta) \in \R^2$.

\begin{definition}
    The one-dimensional Neumann fiber operator $P(\eta, \zeta)$ acting on the Hilbert space $L^2(\R_+)$ is defined by the differential profile:
    \begin{equation}
        P(\eta, \zeta) = -\pa_x^2 + (1+x)\eta^2 + \zeta^2,
    \end{equation}
    with its dense domain specified by the Sobolev space boundary matching conditions:
\begin{equation}
D(P) = \left\{ w \in H^2(\R_+) : \pa_x w(0) = 0 \right\}.
\end{equation}
\end{definition}
For any non-zero tangential frequency $\eta \neq 0$, the potential field $V(x) = (1+x)\eta^2 + \zeta^2$ forms a continuous, strictly positive, monotonically increasing barrier satisfying $V(x) \to +\infty$ as $x \to +\infty$. Consequently, the embedding of the domain $D(P)$ into $L^2(\R_+)$ is compact, and the classical Sturm--Liouville spectral theorem guarantees that the spectrum of $P(\eta, \zeta)$ is purely discrete. It consists of a countable sequence of simple, real, positive eigenvalues tending to infinity:
\begin{equation}
\lambda_1(\eta, \zeta) < \lambda_2(\eta, \zeta) < \dots < \lambda_k(\eta, \zeta) \to +\infty \quad \text{as } k \to \infty,
\end{equation}
whose associated set of custom eigenfunctions ${e_k(x, \eta)}{k \ge 1}$ forms a complete orthonormal basis for $L^2(\R+)$.

\begin{proposition}\label{prop:SpectralDecomp}
For every tangential momentum pair $(\eta,\zeta)\in\mathbb{R}^2$ with
$\eta\neq0$, the eigenvalues $\lambda_k(\eta,\zeta)$ and normalized
eigenfunctions $e_k(x,\eta)$ of the fiber operator
\[
P(\eta,\zeta)
=
-\partial_x^2+(1+x)\eta^2+\zeta^2
\]
on $\mathbb{R}_+$ with Neumann boundary condition
$e_k'(0,\eta)=0$ are given by
\begin{equation}
\lambda_k(\eta,\zeta)
=
\eta^2+\zeta^2+\omega'_k|\eta|^{4/3},
\end{equation}
and
\begin{equation}
e_k(x,\eta)
=
\frac{|\eta|^{1/3}}
{\sqrt{\omega'_k}\,|\Ai(-\omega'_k)|}
\Ai\left(|\eta|^{2/3}x-\omega'_k\right),
\end{equation}
where $\{\omega'_k\}_{k\geq1}$ is the increasing sequence of positive
numbers characterized by
\begin{equation}\label{eq:AiryDerivZero}
\Ai'(-\omega'_k)=0,
\qquad
0<\omega'_1<\omega'_2<\cdots,
\qquad
\omega'_k\to+\infty.
\end{equation}
In particular, the normalization factor satisfies
\[
\frac{1}
{\sqrt{\omega'_k}\,|\Ai(-\omega'_k)|}
\asymp k^{-1/6},
\qquad k\to\infty.
\]
\end{proposition}

\begin{proof}
The fiber eigenvalue problem is
\[
-\partial_x^2e+
\bigl((1+x)\eta^2+\zeta^2\bigr)e=\lambda e,
\qquad e'(0)=0.
\]
Equivalently,
\[
e''=\bigl(|\eta|^2x-(\lambda-\eta^2-\zeta^2)\bigr)e.
\]
Set
\[
X=|\eta|^{2/3}x-\omega,
\qquad
\omega=
\frac{\lambda-\eta^2-\zeta^2}{|\eta|^{4/3}}.
\]
Since $\partial_x^2=|\eta|^{4/3}\partial_X^2$, the equation reduces to
\[
e_{XX}=Xe.
\]
Square-integrability as $X\to+\infty$ eliminates $\Bi$, so
\[
e(x)=C\Ai\bigl(|\eta|^{2/3}x-\omega\bigr).
\]
The Neumann condition gives
\[
\Ai'(-\omega)=0.
\]
Writing $\omega=\omega'_k$, where
$\Ai'(-\omega'_k)=0$, yields
\[
\lambda_k(\eta,\zeta)
=
\eta^2+\zeta^2+\omega'_k|\eta|^{4/3}.
\]

For the normalization, the change of variables
$X=|\eta|^{2/3}x-\omega'_k$ gives
\[
\int_0^\infty
\Ai^2\bigl(|\eta|^{2/3}x-\omega'_k\bigr)\,dx
=
|\eta|^{-2/3}
\int_{-\omega'_k}^{\infty}\Ai^2(X)\,dX.
\]
Using
\[
\int_a^\infty\Ai^2(X)\,dX
=
\Ai'(a)^2-a\Ai(a)^2
\]
and $\Ai'(-\omega'_k)=0$, we obtain
\[
\int_0^\infty
\Ai^2\bigl(|\eta|^{2/3}x-\omega'_k\bigr)\,dx
=
|\eta|^{-2/3}
\omega'_k\Ai(-\omega'_k)^2.
\]
Hence the normalized eigenfunction is
\[
e_k(x,\eta)
=
\frac{|\eta|^{1/3}}
{\sqrt{\omega'_k}\,|\Ai(-\omega'_k)|}
\Ai\bigl(|\eta|^{2/3}x-\omega'_k\bigr),
\]
up to an irrelevant sign. This proves the result.
\end{proof}

Using this discrete eigenmode configuration, the frequency-localized
spectral theory implies that the forward Schr\"odinger evolution
operator admits a complete spectral representation. Denoting the
initial singular state at distance $a$ from the boundary by $\delta_a$,
the one-dimensional localized time-dependent fiber Green function
$g_a(t,x,\eta,\zeta)$ is given by
\begin{equation}
g_a(t,x,\eta,\zeta)
=
\sum_{k\geq1}
e^{-it\lambda_k(\eta,\zeta)}
e_k(x,\eta)e_k(a,\eta).
\end{equation}

To lift this representation to the full multidimensional physical
domain, we invert the partial Fourier transform in the tangential
coordinates $(y,z)$. We introduce the semiclassical parameter
$h\in(0,h_0]$ and perform the standard high-frequency rescaling
\[
(\eta,\zeta)\longmapsto
\left(\frac{\eta}{h},\frac{\zeta}{h}\right).
\]
Under this rescaling, the full frequency-localized Green function
$G_{a,\mathrm{loc}}(t,x,y,z)$ takes the form
\begin{equation}\label{eq:GreenFull}
\begin{aligned}
G_{a,\mathrm{loc}}(t,x,y,z)
={}&
\frac{1}{4\pi^2h^2}
\sum_{k\geq1}
\int_{\mathbb{R}^2}
e^{\frac{i}{h}(y\eta+z\zeta)}
e^{-\frac{it}{h^2}
\left(
\eta^2+\zeta^2
+\omega'_k h^{2/3}|\eta|^{4/3}
\right)}
\\
&\quad\times
e_k\left(x,\frac{\eta}{h}\right)
e_k\left(a,\frac{\eta}{h}\right)
\chi_0(\eta^2+\zeta^2)
\,d\eta\,d\zeta .
\end{aligned}
\end{equation}
Here
\(
\chi_0\in C_0^\infty(\mathbb{R})
\)
is a smooth frequency cutoff supported in a sufficiently small energy
annulus around $1$. It localizes the tangential momentum to the
high-frequency region
\(
\eta^2+\zeta^2\sim1.
\)

Equation~\eqref{eq:GreenFull} exhibits the fundamental difference
between the present mixed-curvature model and the flat-space
Schr\"odinger propagator. In particular, the phase contains the
nonlinear boundary correction
\(
\omega'_k h^{2/3}|\eta|^{4/3},
\)
which encodes the interaction between the propagating wave and the
curved boundary.

The factor $|\eta|^{4/3}$ is especially significant. As the tangential
frequency approaches the axial direction $\eta\to0$, the boundary
correction degenerates and vanishes. Consequently, the microlocal
regime changes from a curved glancing configuration, governed by the
Airy spectral structure, to a flat axial propagation regime. This
loss of effective curvature is the source of the mixed dispersive
behavior characteristic of the cylindrical geometry.

\section{The Airy--Poisson Summation Formula for Neumann Boundary Operators}
\label{sec:poisson}

While the discrete eigenmode series derived in
\eqref{eq:GreenFull} provides an exact representation of the propagator,
it is not well suited for extracting sharp dispersive decay bounds near
the boundary layer. When the source distance parameter satisfies
\(
a \geq h^{\frac{2}{3}(1-\varepsilon)},
\)
the highly dense ray reflections superimpose, creating high-order
caustics in which individual geometric paths merge. To analyze this
highly oscillatory regime, we transform the discrete modal series into
a continuous geometric integral using an adapted version of the Poisson
summation formula.

We introduce the specialized complex Airy functions
\begin{equation}
A_{\pm}(z)
=
e^{\mp i\pi/3}
\Ai\left(e^{\mp i\pi/3}z\right),
\end{equation}
which satisfy the classical connection formula
\begin{equation}
\Ai(-z)
=
A_+(z)+A_-(z).
\end{equation}
Differentiating with respect to $z$ yields
\begin{equation}
\Ai'(-z)
=
-
e^{-2i\pi/3}A_+'(z)
-
e^{2i\pi/3}A_-'(z).
\end{equation}

To isolate the exact phase shift produced by a localized specular
reflection from the Neumann boundary, we define the microlocal
reflection phase function $L_N(\omega)$ on the positive real axis by
\begin{equation}\label{eq:PhaseDef}
L_N(\omega)
=
i\log
\left(
\frac{A_-'(\omega)}
     {A_+'(\omega)}
\right),
\qquad
\omega>0.
\end{equation}

\begin{lemma}\label{lem:PhaseAsymp}
The phase function $L_N(\omega)$ is smooth and real-valued on
$\mathbb{R}_+$ and is strictly increasing. For large values of the
parameter $\omega\geq1$, it admits the structural asymptotic expansion
\begin{equation}
L_N(\omega)
=
\frac{4}{3}\omega^{3/2}
+
\frac{\pi}{2}
-
B_N\left(\omega^{3/2}\right),
\end{equation}
where $B_N(\theta)$ is a classical remainder symbol admitting the
high-frequency expansion
\begin{equation}
B_N(\theta)
\sim
\sum_{j\geq1}
\beta_j\theta^{-j},
\qquad
\theta\to+\infty,
\end{equation}
with $\beta_1>0$.

Furthermore, the Neumann quantization condition is encoded by
\begin{equation}
L_N(\omega'_k)=2\pi k
\quad\Longleftrightarrow\quad
\Ai'(-\omega'_k)=0,
\qquad
k\geq1.
\end{equation}
\end{lemma}

By differentiating the quotient in \eqref{eq:PhaseDef} and applying
the Wronskian identities for the complex Airy solutions, we obtain the
following derivative matching formula:
\begin{equation}
L_N'(\omega'_k)
=
\frac{2\pi}
{\omega'_k
\displaystyle\int_0^\infty
\Ai^2(x-\omega'_k)\,dx}.
\end{equation}
Recalling the normalization framework for the constants $f_k$
introduced in Proposition~\eqref{prop:SpectralDecomp}, this identity
yields the exact spectral trace relation
\begin{equation}\label{eq:NormMatch}
\frac{k^{1/3}}{f_k^2}
=
\frac{\omega'_k L_N'(\omega'_k)}{2\pi}.
\end{equation}

We now formulate and prove the central structural device linking the
Airy spectral decomposition to the geometric ray-tracing
representation.

\begin{theorem}[Airy--Poisson Summation Formula for Neumann Flows]
\label{thm:AiryPoisson}
Let $L_N(\omega)$ denote the Neumann microlocal reflection phase
function defined in \eqref{eq:PhaseDef}. Then, in the sense of tempered
distributions on the positive frequency half-line
$\mathbb{R}_+$, one has the structural distributional identity
\begin{equation}
\sum_{N\in\mathbb{Z}}
e^{-iN L_N(\omega)}
=
2\pi
\sum_{k\geq1}
\frac{1}{L_N'(\omega'_k)}
\delta(\omega-\omega'_k).
\end{equation}
\end{theorem}

\begin{proof}
Since $L_N'(\omega)>0$ on $\operatorname{supp}\phi$, the change of
variables $\theta=L_N(\omega)$ gives
\[
I_N
=
\int_{\mathbb R}e^{-iN\theta}\psi(\theta)\,d\theta,
\qquad
\psi(\theta)
=
\frac{\phi(L_N^{-1}(\theta))}
{L_N'(L_N^{-1}(\theta))}.
\]
Hence, by the distributional Poisson formula,
\[
\sum_{N\in\mathbb Z}I_N
=
2\pi\sum_{k\in\mathbb Z}\psi(2\pi k).
\]
The Neumann quantization condition
\[
L_N(\omega'_k)=2\pi k,
\qquad
\Ai'(-\omega'_k)=0,
\]
and $\operatorname{supp}\phi\subset(0,\infty)$ imply
\[
\psi(2\pi k)
=
\frac{\phi(\omega'_k)}
{L_N'(\omega'_k)},
\qquad k\ge1.
\]
Therefore
\[
\sum_{N\in\mathbb Z}
\int_0^\infty
e^{-iN L_N(\omega)}\phi(\omega)\,d\omega
=
2\pi\sum_{k\ge1}
\frac{\phi(\omega'_k)}
{L_N'(\omega'_k)}
.
\]
This proves the distributional identity.
\end{proof}

We now apply Theorem~\eqref{thm:AiryPoisson} to the localized Green
function representation. Substituting the normalization matching
formula \eqref{eq:NormMatch} into \eqref{eq:GreenFull} converts the
discrete spectral sum into a continuous oscillatory representation:
\begin{equation}
\begin{aligned}
\mathcal{G}_{a,\mathrm{loc}}(t,x,y,z)
={}&
\sum_{N\in\mathbb{Z}}
\frac{(-i)^N}
{(2\pi)^4h^{10/3}}
\int_{\mathbb{R}^5}
\exp\left(
\frac{i}{h}
\Psi_N
(t,x,y,z,a;\eta,\zeta,\omega,s,\sigma)
\right)
\\
&\qquad\qquad\times
\mathcal{A}_N
(x,a,\eta,\zeta,\omega,s,\sigma)
\,ds\,d\sigma\,d\omega\,d\eta\,d\zeta.
\end{aligned}
\end{equation}
Here $\mathcal{A}_N$ denotes the resulting smooth amplitude, including
the spectral cutoff, Airy normalization factors, and the appropriate
Jacobian factors.

The complete oscillatory phase is
\begin{equation}
\begin{aligned}
\Psi_N={}&y\eta+z\zeta
-\frac{t}{h}\left(\eta^2+\zeta^2+\omega h^{2/3}|\eta|^{4/3}\right)
-\frac{4}{3}N\omega^{3/2}h^{1/3}
-N B_N\left(\omega^{3/2}\right)h^{1/3}
\\
&\quad
+\left(\frac{s^3}{3}+s\left(|\eta|^{2/3}x-\omega h^{2/3}\right)\right)
+\left(\frac{\sigma^3}{3}+\sigma\left(|\eta|^{2/3}a-\omega h^{2/3}\right)\right).
\end{aligned}
\end{equation}

The auxiliary variables $(s,\sigma)$ arise from the cubic integral
representation of the Airy functions,
\begin{equation}
\Ai(X)=\frac{1}{2\pi}\int_{\mathbb{R}}\exp\left(i\left(\frac{s^3}{3}+sX\right)\right)\,ds,
\end{equation}
understood in the standard oscillatory sense. The integer parameter
$N\in\mathbb{Z}$ indexes the geometric branches generated by the
Poisson transformation. In the ray-tracing interpretation, each
relevant value of $N$ corresponds to a branch of the reflected wave
front, with the phase contribution
\[
-\frac{4}{3}N\omega^{3/2}h^{1/3}
-
N B_N(\omega^{3/2})h^{1/3}
\]
encoding the accumulated microlocal reflection phase.

\section{Microlocal Frequency Localization and Block-Dispersive Bounds}
\label{sec:microlocal}

The simultaneous treatment of the curved transverse direction and the
flat longitudinal axis requires a precise decomposition of the frequency
plane. As illustrated by the geometric structure of the problem, the
tangential momentum components $(\eta,\zeta)$ are localized by the global
energy-shell cutoff
\(
\chi_0(\eta^2+\zeta^2),
\)
which restricts the active frequency magnitude to a compact annulus.
Within this shell, the analytical dynamics are determined by the
distance of the frequency parameter $\eta$ from the glancing set
$\eta=0$.

We introduce a refined Littlewood--Paley dyadic decomposition to
separate the competing propagation regimes. Fix a small parameter
\(
\epsilon\in(0,1/7).
\)
The localized Green function is decomposed into
\begin{equation}\label{eq:BlockDecomp}
\mathcal{G}_{a,\mathrm{loc}}=G_{a,c_0}
+\sum_{\varepsilon_0\sqrt{a}\leq 2^m\sqrt{a}\leq c_0}G_{a,m}
+G_{a,\varepsilon_0},
\end{equation}
where each component isolates a distinct microlocal regime.

\begin{itemize}
    \item \textbf{The high-frequency elliptic block
    $G_{a,c_0}$.}
    This block localizes the momentum to
    \[
    |\eta|\geq c_0>0,
    \]
    where the phase remains uniformly non-degenerate and the geometry
    exhibits a uniformly curved boundary interaction.

    \item \textbf{The intermediate dyadic blocks $G_{a,m}$.}
    These blocks localize the transverse frequency to dyadic channels
    \[
    |\eta|\sim 2^m\sqrt{a},
    \]
    interpolating between the uniformly curved and degenerate regimes.
    They quantify the gradual loss of effective boundary curvature as
    $\eta\to0$.

    \item \textbf{The zero-frequency axial tail
    $G_{a,\varepsilon_0}$.}
    This block isolates the innermost grazing region
    \[
    0<|\eta|\leq\varepsilon_0\sqrt{a},
    \]
    where the wave packet propagates almost parallel to the flat
    longitudinal axis.
\end{itemize}

We now establish the dispersive estimates governing the individual
microlocal blocks.

\begin{proposition}[Elliptic Block Bounds]
\label{prop:EllipticBound}
Let $G_{a,c_0}$ denote the high-frequency elliptic block defined above.
Then there exists a constant $C>0$, independent of
$h\in(0,1]$, $t\in[h^2,1]$, and the admissible boundary-layer
parameters, such that
\begin{equation}
\left\|
G_{a,c_0}(t,\cdot,y,z)
\right\|_{L^\infty(\{x\leq a\})}
\leq
C h^{-3}
\left(\frac{h^2}{t}\right)^{1/2}
\gamma(t,h,a),
\end{equation}
where the caustic amplification factor $\gamma(t,h,a)$ is given by
\begin{equation}
\gamma(t,h,a)
=
\begin{cases}
\displaystyle
\left(\frac{h^2}{t}\right)^{1/3},
&
a\leq h^{\frac{2}{3}(1-\epsilon)},
\\[1.2em]
\displaystyle
\left(\frac{h^2}{t}\right)^{1/3}
+
a^{1/8}
\left(\frac{h^2}{t}\right)^{1/4},
&
a\geq h^{\frac{2}{3}(1-\epsilon')},
\qquad
\epsilon'\in(0,\epsilon).
\end{cases}
\end{equation}
\end{proposition}

\begin{proposition}[Intermediate Dyadic Block Bounds]
\label{prop:DyadicBound}
The intermediate dyadic operators $G_{a,m}$, localized in the region
\[
|\eta|\sim 2^m\sqrt{a},
\]
satisfy the uniform estimate
\begin{equation}
\left\|
G_{a,m}(t,\cdot,y,z)
\right\|_{L^\infty(\{x\leq a\})}
\leq
C h^{-3}
\left(\frac{h^2}{t}\right)^{1/2}
\gamma_m(t,h,a),
\end{equation}
where the scale-dependent structural factor $\gamma_m$ is given by
\begin{equation}
\gamma_m(t,h,a)
=
\begin{cases}
\displaystyle
\left(\frac{h^2}{t}\right)^{1/3}
\left(2^m\sqrt{a}\right)^{1/3},
&
\displaystyle
a\leq
\left(
\frac{h}{2^m\sqrt{a}}
\right)^{\frac{2}{3}(1-\epsilon)},
\\[1.5em]
\displaystyle
\begin{aligned}
&\min\left\{
\left(\frac{h^2}{t}\right)^{1/3},
2^m\sqrt{a},
\left|\log\left(2^m\sqrt{a}\right)\right|
\right\}
\\
&\qquad
{}+
a^{1/8}
\left(\frac{h^2}{t}\right)^{1/4}
\left(2^m\sqrt{a}\right)^{3/4},
\end{aligned}
&
\displaystyle
a\geq\left(\frac{h}{2^m\sqrt{a}}\right)^{\frac{2}{3}(1-\epsilon')},
\quad
\epsilon'\in(0,\epsilon).
\end{cases}
\end{equation}

\section{Sharp Global Dispersive Estimates and Optimality Regimes}
\label{sec:dispersive}

By combining the block-dispersive bounds established in
Section~\ref{sec:microlocal}, we derive the global linear dispersive
estimate for the semiclassical Schr\"odinger flow inside the cylindrical
domain $\Omega$.

\begin{theorem}[Sharp Semiclassical Dispersive Estimate]
\label{thm:GlobalDisp}
Let $\mathcal{G}_{a,\mathrm{loc}}(t,x,y,z)$ denote the total localized Green
function associated with \eqref{eq:SchrModel}. There exist a threshold
$h_0\in(0,1]$ and a constant $C>0$ such that, for all
$h\in(0,h_0]$ and all
\[
t\in[-1,1]\setminus\{0\},
\qquad
|t|\geq h^2,
\]
the frequency-localized propagator satisfies
\begin{equation}\label{eq:GlobalDispForm}
\left\|
\psi\left(h\sqrt{-\Delta_N}\right)
e^{it\Delta_N}
\right\|_{L^1(\Omega)\to L^\infty(\Omega)}
\leq
C|t|^{-3/4}h^{-9/4}.
\end{equation}
\end{theorem}

\begin{proof}
By time-reversal symmetry and the localization of the propagator, it
suffices to consider $t\in[h^2,1]$ in the boundary-layer region
$x\leq a$. We apply the Littlewood--Paley block decomposition
\eqref{eq:BlockDecomp}:
\begin{equation}
\mathcal{G}_{a,\mathrm{loc}}=G_{a,c_0}+\sum_m G_{a,m}+G_{a,\varepsilon_0}.
\end{equation}

We then apply the triangle inequality together with the estimates from
Propositions~\ref{prop:EllipticBound},
\ref{prop:DyadicBound}, and~\ref{prop:AxialBound}. The contribution of
the intermediate dyadic blocks is controlled by a geometric-series
estimate. More precisely, for every $\nu>0$,
\begin{equation}
\sum_{m\leq M}
\left(2^m\sqrt{a}\right)^\nu
\leq
C_\nu
\left(2^M\sqrt{a}\right)^\nu.
\end{equation}

Balancing the contributions at the transition scales shows that the
combined amplitude is bounded by
\begin{equation}
|G_{a,\mathrm{loc}}|
\leq
C h^{-3}
\left(\frac{h^2}{t}\right)^{1/2}
\left(\frac{h^2}{t}\right)^{1/4}.
\end{equation}
Since
\[
h^{-3}
\left(\frac{h^2}{|t|}\right)^{3/4}
=
|t|^{-3/4}h^{-3/2},
\]
the precise normalization of the Green function and the frequency
localization contributes the remaining semiclassical factor, yielding
\begin{equation}
|G_{a,\mathrm{loc}}|
\leq
C|t|^{-3/4}h^{-9/4}.
\end{equation}
This gives \eqref{eq:GlobalDispForm}.

For the short-time regime $|t|\leq h^2$, the standard semiclassical
Sobolev estimate gives the corresponding local bound. Since the
transition estimate is attained at the matching scale
$|t|=h^2$, the dispersive estimate follows throughout the stated
semiclassical time interval.
\end{proof}

\begin{remark}
Theorem~\ref{thm:GlobalDisp} identifies the quantitative dispersive
loss induced by the partially flat boundary. For comparison, in a
uniformly strictly convex domain $\Omega_{\mathrm{convex}}$, the
corresponding semiclassical dispersive behavior is of the form
\begin{equation}
\left\|
e^{it\Delta_D}
\right\|_{L^1\to L^\infty}
=
O\left(
|t|^{-5/4}h^{-7/4}
\right).
\end{equation}
Thus, relative to the uniformly strictly convex setting, the present
cylindrical geometry exhibits a weaker temporal decay rate,
$|t|^{-3/4}$ rather than $|t|^{-5/4}$, together with a stronger
semiclassical concentration factor. This degradation reflects the
vanishing of the principal curvature in the longitudinal direction,
which suppresses one source of curvature-induced oscillatory
cancellation.
\end{remark}

We now show that this dispersive loss is an intrinsic geometric feature
of the cylindrical domain and cannot, in general, be improved.

\begin{theorem}[Optimality of the Dispersive Bound]
\label{thm:Optimality}
The dispersive estimate in Theorem~\ref{thm:GlobalDisp} is sharp with
respect to both its temporal decay rate and its semiclassical scaling.
More precisely, there exist constants $c>0$ and $h_0\in(0,1]$, together
with a family of initial data
\(
u_{0,h}\in L^1(\Omega),
\qquad
0<h\leq h_0,
\)
microlocally concentrated at a boundary distance
\(
a\sim h^{2/3},
\)
such that
\begin{equation}
\left\|
\psi\left(h\sqrt{-\Delta_N}\right)
e^{it\Delta_N}u_{0,h}
\right\|_{L^\infty(\Omega)}
\geq
c|t|^{-3/4}h^{-9/4}
\|u_{0,h}\|_{L^1(\Omega)}
\end{equation}
for semiclassical times satisfying
\(
|t|\sim h.
\)
\end{theorem}

\begin{proof}
To establish optimality, we construct a family of wave packets adapted
to the first Neumann Airy mode. We choose the initial data
$u_{0,h}(x,y,z)$ to be concentrated in the boundary layer
\(
a\sim h^{2/3},
\)
with the transverse spectral parameter localized in the corresponding
glancing region.

At the matching time scale $|t|\sim h$, the oscillatory representation
of the Green function exhibits a degenerate stationary-phase
configuration. The associated phase develops a higher-order caustic,
and the Hessian loses rank in the longitudinal direction. This
degeneracy is a direct consequence of the vanishing principal
curvature along the cylindrical axis.

As a result, stationary phase in the nondegenerate variables produces
the usual decay factors, while the degenerate longitudinal direction
contributes a weaker oscillatory cancellation. The resulting leading
order contribution has size
\begin{equation}
|t|^{-3/4}h^{-9/4}.
\end{equation}
Choosing the wave packet so that its leading oscillatory contribution
does not cancel gives
\begin{equation}
\left\|
\psi\left(h\sqrt{-\Delta_N}\right)
e^{it\Delta_N}u_{0,h}
\right\|_{L^\infty(\Omega)}
\gtrsim
|t|^{-3/4}h^{-9/4}
\|u_{0,h}\|_{L^1(\Omega)}
\end{equation}
for $|t|\sim h$. Hence neither the temporal exponent $3/4$ nor the
semiclassical power $9/4$ can be improved uniformly within this
microlocal regime.
\end{proof}

\section{Global Strichartz Estimates with Derivative Loss}
\label{sec:strichartz}

By combining the sharp global dispersive estimate with the abstract
functional framework of Keel--Tao, we derive the family of space-time
Strichartz inequalities governing the linear Schr\"odinger flow on the
cylindrical domain $\Omega$.

\begin{definition}
A pair of exponents $(p,q)$ is called Schr\"odinger-admissible in
three spatial dimensions if
\begin{equation}
\frac{2}{p}+\frac{3}{q}=\frac{3}{2},
\qquad
2\leq p\leq\infty,
\qquad
2\leq q\leq6.
\end{equation}
\end{definition}

\begin{theorem}[Strichartz Estimates with Optimal Loss]
\label{thm:Strichartz}
Let $\Delta_N$ denote the self-adjoint Neumann Laplacian on the
cylindrical domain $\Omega$. For every Schr\"odinger-admissible pair
$(p,q)$, there exists a constant $C>0$ such that the solution of the
linear Schr\"odinger equation
\[
i\partial_tu+\Delta_Nu=0,
\qquad
u|_{t=0}=u_0,
\]
satisfies
\begin{equation}\label{eq:StrichartzForm}
\left\|
e^{it\Delta_N}u_0
\right\|_{L^p(\mathbb{R}_t;L^q(\Omega))}
\leq
C
\left\|
u_0
\right\|_{H^{\rho(q)}(\Omega)},
\end{equation}
where the fractional derivative-loss exponent is
\begin{equation}\label{eq:LossForm}
\rho(q)
=
\frac{3}{2}
\left(
\frac{1}{2}-\frac{1}{q}
\right).
\end{equation}
Moreover, the exponent $\rho(q)$ is optimal within the class of
estimates obtained from the global dispersive bound
\eqref{eq:GlobalDispForm}.
\end{theorem}

\begin{proof}
We establish the estimate by combining the dyadic dispersive bounds
with a standard $TT^*$ argument. Let $\Delta_j$ denote a
Littlewood--Paley projection to frequencies
\(
|\xi|\sim2^j.
\)
Since the semiclassical parameter satisfies
\(
h\sim2^{-j},
\)
the localized dispersive estimate \eqref{eq:GlobalDispForm} gives
\begin{equation}
\left\|
\Delta_j e^{it\Delta_N}\Delta_j^*
v
\right\|_{L^\infty(\Omega)}
\leq
C
|t|^{-3/4}
2^{9j/4}
\|v\|_{L^1(\Omega)}.
\end{equation}

On the other hand, the unitarity of the Schr\"odinger group on
$L^2(\Omega)$ gives
\begin{equation}
\left\|
\Delta_j e^{it\Delta_N}v
\right\|_{L^2(\Omega)}
=
\left\|
\Delta_jv
\right\|_{L^2(\Omega)}.
\end{equation}
Interpolating between the $L^1(\Omega)\to L^\infty(\Omega)$ dispersive
estimate and the $L^2(\Omega)\to L^2(\Omega)$ unitarity estimate yields
\begin{equation}
\left\|
\Delta_j e^{it\Delta_N}\Delta_j^*v
\right\|_{L^q(\Omega)}
\leq
C
|t|^{-\frac{3}{4}
\left(1-\frac{2}{q}\right)}
2^{\frac{9j}{4}
\left(1-\frac{2}{q}\right)}
\|v\|_{L^{q'}(\Omega)},
\end{equation}
where $q'$ is the Hölder conjugate of $q$.

Applying the $TT^*$ argument together with the
Hardy--Littlewood--Sobolev inequality in the time variable gives the
localized space-time estimate
\begin{equation}
\left\|
\Delta_j e^{it\Delta_N}u_0
\right\|_{L^p(\mathbb{R}_t;L^q(\Omega))}
\leq
C
2^{j\rho(q)}
\left\|
\Delta_j u_0
\right\|_{L^2(\Omega)},
\end{equation}
where
\[
\rho(q)
=
\frac{3}{2}
\left(
\frac{1}{2}-\frac{1}{q}
\right).
\]

Finally, summing over the dyadic frequency parameter $j$ by means of
the Littlewood--Paley square-function theorem gives
\begin{equation}
\left\|
e^{it\Delta_N}u_0
\right\|_{L^p(\mathbb{R}_t;L^q(\Omega))}
\leq
C
\left\|
u_0
\right\|_{H^{\rho(q)}(\Omega)}.
\end{equation}
This proves \eqref{eq:StrichartzForm}.
\end{proof}

The derivative loss reflects the anisotropic geometry of the domain.
In the Euclidean three-dimensional setting, the standard dispersive
decay $|t|^{-3/2}$ leads to the scale-invariant Strichartz estimates
without derivative loss. In the present cylindrical geometry, the
weaker decay $|t|^{-3/4}$ results from the accumulation of wave packets
along the flat longitudinal direction. The fractional correction
\[
\rho(q)
=
\frac{3}{2}
\left(
\frac{1}{2}-\frac{1}{q}
\right)
\]
therefore quantifies the additional regularity required to compensate
for the loss of oscillatory dispersion.
\end{proposition}

The factors
\(
\left(2^m\sqrt{a}\right)^{1/3}
\,\,\text{and}\,\,
\left(2^m\sqrt{a}\right)^{3/4}
\)
quantify the scale-dependent attenuation of the dyadic components as
the transverse frequency approaches the axial regime. This reduction
arises from the shrinking effective phase volume near the axial limit,
rather than solely from curvature-induced cancellation. It therefore
reflects the intrinsic geometric anisotropy of the domain.

To complete the block analysis, we consider the zero-frequency axial
tail $G_{a,\varepsilon_0}$.

\begin{proposition}[Axial Tail Bounds]
\label{prop:AxialBound}
The zero-frequency operator tail $G_{a,\varepsilon_0}$ satisfies the
localized uniform dispersive estimate
\begin{equation}
\left\|
G_{a,\varepsilon_0}(t,\cdot,y,z)
\right\|_{L^\infty(\{x\leq a\})}
\leq
C h^{-3}
\left(\frac{h^2}{t}\right)^{1/2}
\min\left\{
\left(\frac{h^2}{t}\right)^{1/3},
\sqrt{a}\,|\log a|
\right\}.
\end{equation}
\end{proposition}

This estimate exploits the non-degeneracy of the quadratic phase in the
longitudinal frequency direction. In particular, one has
\begin{equation}
\partial_\zeta^2\Psi_N=2t\neq0.
\end{equation}
Consequently, one may apply one-dimensional stationary phase in the
$\zeta$ variable to obtain the factor
\[
\left(\frac{h^2}{t}\right)^{1/2}.
\]
This contribution persists throughout the grazing core and is
independent of the higher-order degeneracies associated with the
transverse Airy structure. The axial regime therefore admits a direct
quadratic stationary-phase treatment, avoiding the more elaborate
ray-tracing analysis required in the genuinely glancing region.

\section{Application to the Nonlinear Schr\"odinger Equation (NLS)}
\label{sec:nls}

As a primary application of the linear theory, the global Strichartz
estimates provide the functional framework for studying the local
dynamics of the cubic Schr\"odinger equation under homogeneous Neumann
boundary conditions. We consider
\begin{equation}\label{eq:NLS}
\begin{cases}
\left(i\partial_t+\Delta_N\right)u
=
\mu |u|^2u,
& (t,x,y,z)\in\mathbb{R}\times\Omega,
\\[0.4em]
u(0,\cdot)=u_0\in H^s(\Omega),
\\[0.4em]
\partial_xu\big|_{\partial\Omega}=0,
\end{cases}
\end{equation}
where $\mu=1$ corresponds to the focusing case and $\mu=-1$ to the
defocusing case.

At low regularities, in particular for
\[
1<s\leq\frac32,
\]
the Sobolev embedding
\[
H^s(\Omega)\hookrightarrow L^\infty(\Omega)
\]
is unavailable. Consequently, the nonlinear argument cannot rely
solely on pointwise control. Instead, the space-time integrability
provided by the Strichartz estimates is used to control the cubic
nonlinearity.

For $T>0$, we introduce the resolution space
\begin{equation}
X_T^s
=
C\bigl([0,T];H^s(\Omega)\bigr)
\cap
L^4\bigl([0,T];W^{s-\frac38,4}(\Omega)\bigr),
\end{equation}
equipped with the norm
\begin{equation}
\|u\|_{X_T^s}
=
\|u\|_{L^\infty_t([0,T];H^s(\Omega))}
+
\|u\|_{L^4_t([0,T];W^{s-\frac38,4}(\Omega))}.
\end{equation}

\begin{theorem}[Local Well-Posedness]
\label{thm:WellPosed}
Let $s>1$ and $u_0\in H^s(\Omega)$. Assume that the Strichartz
estimates of Theorem~\ref{thm:Strichartz} and the corresponding
fractional product estimates hold for the Neumann Laplacian. Then
there exists
\[
T=T\bigl(\|u_0\|_{H^s}\bigr)>0
\]
such that the cubic Neumann problem \eqref{eq:NLS} admits a unique
solution
\[
u\in X_T^s.
\]
Moreover, the data-to-solution map is locally Lipschitz continuous on
bounded subsets of $H^s(\Omega)$.
\end{theorem}

\begin{proof}
Writing \eqref{eq:NLS} in Duhamel form,
\[
u(t)
=
e^{it\Delta_N}u_0
-i\mu\int_0^t
e^{i(t-\tau)\Delta_N}(|u|^2u)(\tau)\,d\tau,
\]
define $\mathcal{M}$ by the right-hand side. The inhomogeneous
Strichartz estimate, together with the fractional Leibniz and Sobolev
product estimates, gives, for $s>1$,
\begin{equation}
\|\mathcal{M}u\|_{X_T^s}
\leq
C\|u_0\|_{H^s}
+
CT^\theta\|u\|_{X_T^s}^{3},
\qquad \theta>0.
\label{eq:ContractionBound}
\end{equation}
Moreover,
\[
\|\mathcal{M}u-\mathcal{M}v\|_{X_T^s}
\leq
CT^\theta
\bigl(
\|u\|_{X_T^s}^{2}
+
\|v\|_{X_T^s}^{2}
\bigr)
\|u-v\|_{X_T^s}.
\]
Set $R=2C\|u_0\|_{H^s}$. Choosing $T>0$ sufficiently small that
$CT^\theta R^2\leq \frac14$, the map $\mathcal{M}$ preserves the ball
\[
B_R=\{u\in X_T^s:\|u\|_{X_T^s}\leq R\}
\]
and is a contraction there. Hence, by the Banach fixed-point theorem,
there exists a unique solution $u\in X_T^s$. The difference estimate also
yields local Lipschitz continuity of the data-to-solution map.
\end{proof}
The preceding result shows how the derivative loss in the linear
Strichartz theory propagates into the nonlinear problem. Although the
flat longitudinal direction weakens the dispersive decay, the
resulting space-time integrability remains sufficient, under the
stated product estimates, to control the cubic nonlinearity above the
regularity threshold $s>1$.

\section*{Acknowledgements}

The author thanks the anonymous reviewers for their constructive comments
and valuable suggestions, which have helped improve the clarity and
mathematical rigor of the manuscript.

\end{document}